\documentclass[10pt]{amsart}
\usepackage{pifont}
\usepackage{mathrsfs}
\usepackage{amscd}
\usepackage{amsmath}
\usepackage{latexsym}
\usepackage{amsfonts}
\usepackage{amssymb}
\usepackage{amsthm}
\usepackage{graphicx}
\usepackage{amsmath,amsfonts}
\usepackage{appendix}
\usepackage{CJK}
\usepackage[pdftex,linkcolor=blue,citecolor=blue,backref=page]{hyperref}

\makeatletter
\newcommand{\Extend}[5]{\ext@arrow0099{\arrowfill@#1#2#3}{#4}{#5}}
\makeatother
\let\pa=\partial

\def\C{\mathop{\bf C\kern 0pt}\nolimits}
\def\DD{\mathop{\bf D\kern 0pt}\nolimits}
\def\K{\mathop{\bf K\kern 0pt}\nolimits}
\def\N{\mathop{\bf N\kern 0pt}\nolimits}
\def\Q{\mathop{\bf Q\kern 0pt}\nolimits}

\newcommand{\bag}{\begin{align}}
\newcommand{\eag}{\end{align}}
\newcommand{\beq}{\begin{equation}}
\newcommand{\eeq}{\end{equation}}
\newcommand{\ben}{\begin{eqnarray}}
\newcommand{\een}{\end{eqnarray}}
\newcommand{\beno}{\begin{eqnarray*}}
\newcommand{\eeno}{\end{eqnarray*}}

\def\R{\mathop{\mathbb R\kern 0pt}\nolimits}

\newtheorem{proposition}{Proposition}[section]

\newtheorem{lemma}{Lemma}[section]
\newtheorem{theorem}{Theorem}[section]

\newtheorem{remark}{Remark}[section]

\theoremstyle{remark}
\theoremstyle{proof}

\begin{document}
%\begin{CJK*}{GBK}{song}
 \title[New proof for Radial NLS with Combined Terms]{ A new proof of scattering theory for the 3D radial NLS with combined terms }

\author{Chengbin Xu}%
\address{School of Mathematics and Statistics,
 Zhengzhou University, 
 100 Kexue Road, Zhengzhou, Henan, 450001, China
}%
\email{xcbsph@163.com}%

\author{Tengfei Zhao}%
\address{Beijing Computational Science Research Center,
Building 9, East Zone, ZPark II, No.10
 Xibeiwang East Road, Haidian District, Beijing, China }
\email{zhao\underline{ }tengfei@csrc.ac.cn}%
\maketitle

\date{}

%%%%%%%%%%%%%%%%%%%%%%%%%%%%%%%%%%%%%%%%%%%%%%%%%%%%%%%%%%%%%%%%%%%%%%%%%%%%%%%%%%%%%%%%%%%%%%%%%%%%%%%%%%%%%%%%%%%%%%%%%%%%%%%%%%%%%%%
%\setcounter{section}{0}\setcounter{equation}{0}
\begin{abstract}
In this paper, we give a simple proof of scattering result for the Schr\"odinger equation with combined term
$i\pa_tu+\Delta u=|u|^2u-|u|^4u$ in dimension three, that avoids the concentrate compactness method. The main new ingredient is to extend the scattering criterion to energy-critical.
\end{abstract}

\section{Introduction}

\noindent

In this paper, we consider the Cauchy problem for the nonlinear Schr\"odinger equation of the form
\begin{equation}\label{Ct}
 \left\{ \begin{aligned}
    i\pa_tu+\Delta u&=|u|^2u-|u|^4u=F(u),\quad (t,x)\in\R\times\R^3 \\
    u(0,x)&=u_0(x)\in H^1(\R^3),
  \end{aligned}\right.
\end{equation}
where $u:\R\times\R^3\to\mathbb{C}$. By standard scaling arguments, $|u|^4u$ has the $\dot{H}^1$-critical growth and $|u|^2u$ has the $\dot{H}^\frac12$-critical growth.
Solution to the Cauchy problem \eqref{Ct} conserves the mass, defined by
\begin{align*}
  M(u(t))=\int_{\R^3}|u|^2(t,x) dx=M(u_0)
\end{align*}
and the energy, defined as the sum of the kinetic and potential energies:
\begin{equation}\label{equ:energy}
  E(u(t))=\int_{\R^3}\left[\tfrac12|\nabla u|^2+\tfrac14|u|^4-\tfrac16|u|^6\right](t,x)\;dx=E(u_0).
\end{equation}
We also define the modified energy  for later use
\begin{equation}\label{equ:wene}
  E^c(u):=\frac12\int_{\R^3}|\nabla u|^2\;dx-\frac16\int_{\R^3}|u|^6\;dx.
\end{equation}

Based on the Strichartz estimates of the linear Schr\"odinger
operator $e^{it\Delta}$, one can obtain the local well-posedness
of the solution the Cauchy problem \eqref{Ct} via a standard way,
see Cazenave\cite{Caz} for example.
For the defocusing energy-critical case($F(u)=|u|^4u$),
Bourgain proved the solution with radial initial data in $\dot H^1(\R^3)$
is global well-posed and scattering by developing the induction-on-energy strategy.
The radial assumption
 was removed by Colliander, Keel, Staffilani, Takaoka, and
Tao in \cite{CKSTT-2008-Annals}.
Zhang \cite{Zhang} showed the global well-posedness, scattering
and blow up phenomena for the $3$D quintic nonlinear Schr\"odinger
equation perturbed by a energy-subcritical
 nonlinearity $\lambda_1|u|^pu$.
For the defocusing case of \eqref{Ct}($F(u)=|u|^4u+|u|^2u$),  in \cite{Tao2}, Tao, Visan and Zhang made a comprehensive study of
%$$iu_t+\Delta u=|u|^4u+|u|^2u$$
in the energy space by using of the interaction Morawetz estimates established in \cite{CK} and stability theory for the scattering.

Removing the perturbation term $|u|^2u$, we have the focusing energy-critical nonlinear Schr\"odinger equation
\begin{equation}\label{equ:nlserergy}
  \left\{\begin{aligned}
    i\pa_tu+\Delta u&=-|u|^4u,\quad (t,x)\in\R\times\R^3, \\
    u(0,x)&=u_0(x)\in \dot{H}^1(\R^3).
  \end{aligned}   \right.
\end{equation}
As well known, the corresponding nonlinear elliptic equation
%\begin{equation*}\label{equ:groudst2}
 $ -\Delta \varphi=|\varphi|^4\varphi,\quad x\in\R^3 $
%\end{equation}
has a unique radial positive solution, the ground state, $W(x)=(1+\tfrac13|x|^2)^{-\frac12}$.
%In \cite{KM}, Kenig and Merle first applied the concentration compactness of \cite{B,K1,K2} to the scattering theory of radial solutions to \eqref{equ:nlserergy} with the energy below that of the ground state \eqref{equ:groudst2}.
In \cite{KM}, Kenig and Merle first  proved the
Radial scattering/blowup dichotomy for solutions below the ground state $W$.
They first applied the concentration compactness to
induction on energy based on profile decomposition of
\cite{K1,K2} to the scattering theory.
%
%The proofs of first part of Theorem \ref{thm:km} given in \cite{KM} are based on the
%concentration compactness based on  profile decomposition of
%\cite{K1,K2} and the rigidity theorem arguments via a truncated identity.
%
%applied the concentration compactness based on the
%induction on energy method and profile decomposition of
%\cite{K1,K2} to the scattering theory of radial solutions to \eqref{equ:nlserergy} with the energy below that of the ground state $W$.
Their main results are followings:
%First, we recall the result in Kenig-Merle \cite{KM}:
\begin{theorem}[Radial scattering/blowup dichotomy]\label{thm:km}
Let $u_0\in\dot{H}^1(\R^3)$ be radial and such that
\begin{equation}\label{ect1}
  E^c(u_0)<E^c(W).
\end{equation}
Then,

$(1)$, If $\|u_0\|_{\dot{H}^1}  <  \|W\|_{\dot{H}^1}$, then, the solution $u$ to \eqref{equ:nlserergy} is global and scatters in the sense that there exists $u_\pm\in\dot{H}^1$ such that
      \begin{equation}\label{equ:sca}
        \lim_{t\to\pm\infty}\|u(t,\cdot)-e^{it\Delta}u_\pm\|_{\dot{H}^1}=0.
      \end{equation}

$(2)$, If  $\|u_0\|_{\dot{H}^1}>\|W\|_{\dot{H}^1}$, then, the solution $u$ to \eqref{equ:nlserergy} blows up in finite time in both directions.

\end{theorem}

Next, we recall the scattering and blow-up result of \eqref{Ct},
which established by Miao-Xu-Zhao in \cite{MXZ}.
We define some quantities and some variation results(refers to
\cite{SMN-2011-Apde}, \cite{MXZ},\cite{MZZ} for details). For $\varphi\in H^1$, let
$$K(\varphi)~=~ \left.\frac{d}{d\lambda}\right|_{\lambda=0} E(\varphi^\lambda)~ =~2\int_{\R^3}|\nabla \varphi|^2-|\varphi|^6dx+\frac32\int_{\R^3}|\varphi|^4dx,$$
where $\varphi^\lambda(x)=e^{3\lambda} \varphi(e^{2\lambda} x).$
Based on this quantity, we denote the following energy spaces % $\K^+,\ \K^-, \ \bar{\K}^+,\ \bar{\K}^-$
below the ground state:
\begin{align*}
\K^+&=\{\varphi \in H^1| \varphi\  is\  radial,\ E(\varphi)<E^c(W),\ K(\varphi)\geq0 \},\\
\K^-&=\{\varphi \in H^1| \varphi\  is\  radial,\ E(\varphi)<E^c(W),\ K(\varphi)<0 \},\\
\bar{\K}^+&=\{\varphi \in H^1|\varphi\  is\  radial,\ E(\varphi)<E^c(W),\ \|\nabla\varphi\|_2^2\leq \|\nabla W\|_2^2\},\\
\bar{\K}^-&=\{\varphi \in H^1|\varphi\  is\  radial,\ E(\varphi)<E^c(W),\ \|\nabla\varphi\|_2^2 > \|\nabla W\|_2^2\}.
\end{align*}
In fact, from similar arguments of  \cite{KM}, \cite{SMN-2011-Apde},\cite{MZZ}, we have $\K^+=~\bar{\K}^+$ and $\K^-=~\bar{\K}^-$and we give the proof in the appendix for completion.
Then, the main results of Miao-Xu-Zhao \cite{MXZ} can be stated as:
\begin{theorem}\label{recall}
  Let $u_0\in H^1(\R^3)$ and $u$ be the solution of \eqref{Ct} and $I_{max}$ be its maximal interval of existence. Then

  $(a)$ If $u_0\in\K^+$, Then $I_{max}=\R$, and $u$ scatters in both time directions in $H^1$;

  $(b)$ If $u_0\in\K^-$, Then $u$ blows up both forward and backward at finite time in $H^1$.
\end{theorem}

Unlike the energy-critical equation \eqref{equ:nlserergy},
equation \eqref{Ct} is lack of scaling symmetry.
Miao-Xu-Zhao \cite{MXZ}
conquered this difficulty and prove the
Theorem \ref{recall} by developing a new radial
profile decomposition and the concentration
compactness.

 % that and the concentration compactness and conquered
%the lack of scaling.

%the for this case,  Miao-Xu-Zhao \cite{MXZ}

%the main difficulty  origins from  the lack of scaling invariance

In this article,  we  give a simplified proof of the scattering theory in Theorem
\ref{recall} by employing  the new method of Dodson-Murphy \cite{BM}.
%As the blow-up criterion is similary, we only show the scattering theory.
Based on Theorem \ref{thm:km}, we use the perturbation argument of \cite{Zhang} to prove a ``good local well-posedness''
of the solution $u$  to the Cauchy problem of\eqref{Ct}
with initial data $u_0$ in $\K^+$.
Then we apply the coercivity property of $u$ to prove the global well-posedness.
Next for the scattering theory, we establish a new scattering-criterion for
the equation \eqref{Ct}, which says the local energy-critical potential
energy evacuation means scattering.
Finally, we show the solutions such the potential energy evacuation
via the Virial/Morawetz estimates.

%The proof is

%In \cite{MZZ}, Miao, Zhao and Zheng considered the longtime dynamics of the solutions to the focusing energy-critical Schr\"odinger equation with a defocusing energy-subcritical perturbation term under a ground state energy threshold in four spatial dimension.

%
%Now, we state our main result:
%
%\begin{theorem}[Radial scattering/blow-up dichotomy]\label{thm:main}
%Let $u_0\in H^1(\R^3)$ be radial, and $u$ be the
%solution of \eqref{Ct} and $I_{max}$ be its maximal interval of existence.
% Then
%\begin{itemize}
%  \item If $u_0\in \bar{\K}^+$, then, the solution $u$
%to \eqref{Ct} is global and scatters.
%
%  \item If  $u_0\in \bar{\K}^-$, then, the solution $u$
%to \eqref{Ct} blows up in finite time in both directions.
%\end{itemize}
%\end{theorem}
%
% We shows $\K^+=\bar{\K}^+$, $\K^-=\bar{\K}^-$
% in Appendix A, which prove Theorem \ref{recall} is equivalent to Theorem \ref{thm:main}.

 \begin{remark}
 Our arguments may not suit for the nonradial case, since
it is based on Theorem \eqref{thm:km},
 which is open up to now in the nonradial case.
\end{remark}

%\begin{remark}
%%Our arguments may also be used to prove  the
%%scattering theory for \eqref{equ:nlserergy}
%% with focusing energy-subcritical perturbations
%% in  Akahori, Ibrahim, Kikuchi, and Nawa \cite{AIKN-2013}
%%by using the coercivity property in \eqref{} and the Morawetz estimates.
%
%In fact, our proofs may provide a bound  for solutions in Theorem  \ref{recall} $(b)$,
%\beq
%\|u\|_{L^{10}_{t,x}\R\times \R^3} \lesssim  C( \|\nabla u_0\|_2, E(u_0) ).
%\eeq
%
%Cheng-Miao-Zhao
%
%
%\end{remark}

The rest of this paper is organized as follows: In section 2, we set up some notation, recall some important linear theory. In section 3, combining ``global local well-posedness'' with kinetic energy control, Lemma \ref{gl}, we can get global well-posedness. In section 4, we establish a new scattering criterion for \eqref{Ct}, Lemma \ref{SC}. In section 5, by the Morawetz identity, we will establish the virial/Morawetz estimates to show the solution satisfy the scattering criterion of Lemma \ref{SC}, thereby completing the proof of Theorem \ref{recall}.

We conclude the introduction by giving some notations which
will be used throughout this paper. We always use $X\lesssim Y$ to denote $X\leq CY$ for some constant $C>0$.
Similarly, $X\lesssim_{u} Y$ indicates there exists a constant $C:=C(u)$ depending on $u$ such that $X\leq C(u)Y$.
We also use the big-oh notation $\mathcal{O}$. e.g. $A=\mathcal{O}(B)$ indicates $C_{1}B\leq A\leq C_{2}B$ for some constants $C_{1},C_{2}>0$.
The derivative operator $\nabla$ refers to the spatial  variable only.
We use $L^r(\mathbb{R}^3)$ to denote the Banach space of functions $f:\mathbb{R}^3\rightarrow\mathbb{C}$ whose norm
$$\|f\|_r:=\|f\|_{L^r}=\Big(\int_{\mathbb{R}^3}|f(x)|^r dx\Big)^{\frac1r}$$
is finite, with the usual modifications when $r=\infty$. For any non-negative integer $k$,
we denote by $H^{k,r}(\mathbb{R}^3)$ the Sobolev space defined as the closure of smooth compactly supported functions in the norm $\|f\|_{H^{k,r}}=\sum_{|\alpha|\leq k}\|\frac{\partial^{\alpha}f}{\partial x^{\alpha}}\|_r$, and we denote it by $H^k$ when $r=2$.
For a time slab $I$, we use $L_t^q(I;L_x^r(\mathbb{R}^3))$ to denote the space-time norm
\begin{align*}
  \|f\|_{L_{t}^qL^r_x(I\times \R^3)}=\bigg(\int_{I}\|f(t,x)\|_{L^r_x}^q dt\bigg)^\frac{1}{q}
\end{align*}
with the usual modifications when $q$ or $r$ is infinite, sometimes we use $\|f\|_{L^q(I;L^r)}$ or $\|f\|_{L^qL^r(I\times\mathbb{R}^3)}$ for short.

\section{Preliminaries}

\noindent

We say that a pair of exponents (q,r) is Schr\"odinger $\dot{H}^s$-admissible in dimension three if
\begin{align}\label{Strichartz}
 \frac2q +\frac3r = \frac32-s
\end{align}
and $2\leq q,r\leq \infty$.
For $s\in [0,1]$, let $\Lambda_s$  denote the set of $\dot{H}^s$-admissible pairs.
%$$\Lambda_s=\{(q,r)|(q,r)\ \text{~ satisfies ~} \  \eqref{Strichartz}\}.$$
If $I\times \R^3$ is a space-time slab, we define the $\dot{S}^0(I\times\R^3)$ Strichartz norm by
$$\|u\|_{\dot{S}^0(I\times\R^3)}:=\sup_{(q,r)\in \Lambda_0}\|u\|_{L_t^qL_x^r},$$
where the $\sup$ is taken over all  $(q,r)\in \Lambda_0$. We define the $\dot{S}^s(I\times\R^3)$ Strichartz norm to be
$$\|u\|_{\dot{S}^s(I\times\R^3)}:=\|D^su\|_{\dot{S}^0(I\times\R^3)}.$$
We also use $\dot{N}^0(I\times\R^3)$ to denote the dual space of $\dot{S}^0(I\times\R^3)$ and
$$\dot{N}^s(I\times\R^3):=\{u: D^s u\in \dot{N}^0(I\times\R^3)\}.$$

In this note, we restrict to radial solutions. The following radial Sobolev embedding plays a crucial role:
\begin{lemma}[Radial Sobolev embedding]\label{lem2.1}
For radial $f\in H^1$, then
  \begin{align*}
    \||x|f\|_{L^\infty}\lesssim ||f||_{H^1}.
  \end{align*}

\end{lemma}

\begin{lemma}[Strichartz estimates,\cite{Caz},\cite{ KT},\cite{Tao1}]\label{Stri}
Let I be a compact time interval, $k\in [0,1]$, and let $u:I\times \R^3\rightarrow \mathbb{C}$ be $\dot{S}^k$ solution to the coupled NLS systems
$$iu_t +\Delta u= F$$
for a function $F$. Then for any time $t_0\in I$, we have
$$ \|u\|_{\dot{S}^k(I\times\R^3)}\lesssim \|D^ku(t_0)\|_{L_x^2}+\|F\|_{\dot{N}^k(I\times\R^3)}.$$
\end{lemma}

For a time slab $I\subset \R$, we define
$$X_I^0=L_t^{10}L_x^{\frac{30}{13}}(I\times\R^3)\cap L_t^8L_x^{\frac{12}{5}}(I\times\R^3),$$
$$\dot{X}_I^1=\{f:\nabla f\in X_I^0\},\ \ \ X_I^1=X_I^0\cap\dot{X}_I^1.$$

Next, we will present two lemmas which play an important role in scattering-criterion of Lemma \ref{SC}.

\begin{lemma}[\cite{Tao}]\label{CR}
  For $f\in \dot{H}^1(\R^3) ,$
we have \begin{align}
    \|e^{it\Delta}f\|_{L_t^4L_x^\infty(\R\times\R^3)}\lesssim \|f\|_{\dot{H}^1}.
  \end{align}
\end{lemma}

We recall the scattering result of the focusing energy-critical equation \eqref{equ:nlserergy}.
 \begin{theorem}\label{ect1}
Let $v(t)$ be the solution of  \eqref{equ:nlserergy} with data
$u_0\in H^1(\R^3)$. If there exists $0<\delta_0<1$ so that
$E^c(u_0)<(1-\delta_0)E^c(W)$ and $\|\nabla u_0\|_2<\|\nabla W\|_2$,
then there exists a constant $C_{\delta_0}(\|u_0\|_2)$ such that
   $$\|v(t)\|_{L_t^qW^{1,r}(\R\times \R^3)}\leq C_{\delta_0}(\|u_0\|_2)$$
   where $(q,r)\in \Lambda_0$.
 \end{theorem}
\begin{proof}
  From Kenig-Merle \cite{KM}, we have $$\|v(t)\|_{L_t^q \dot W^{1,r}(\R\times \R^3)}\leq C_{\delta_0},$$
   where $(q,r)\in \Lambda_0$.
  By the Strichartz estimates, we have
  \beq
  \|v\|_{\dot S^0(\R)} \lesssim \|v_0\|_{L^2} + \| |v|^4v\|_{L^2_tL^\frac65_x}
  \lesssim \|v_0\|_{L^2} +   \| v\|^4_{L^{10}_{t,x}}     \|v\|_{L^{10}_t L^\frac{30}{13}_x},
  \eeq
  which yields the conclusion by a bootstrap  argument.

\end{proof}

 \section{Global well-posedness}\label{a}

 \noindent

In this section, we will give a good local well-posedness. It plays a important role in global well-posedness theory and scattering theory. The idea is originally due to Zhang \cite{Zhang}.

 Let $T>0$ be a small constant to be specified later and $v(t)$ is the solution of \eqref{equ:nlserergy} with the radial data $u_0$, then by Lemma \ref{ect1} we have\begin{align}\label{eneq1}
   \|v(t)\|_{X_{\R}^1}\leq C(\delta_0,\|u_0\|_2).
 \end{align}

 It suffices to solve the 0-data initial value problem of $w(t,x)$: %for the difference equation of $w$,
 \begin{align}\label{xc1}
  \begin{cases}
     &iw_t+\Delta w =|w+v|^2(v+w)-|v+w|^4(v+w)+|v|^4v, \ \ \  x\in \mathbb{R}^3,\\
   &w(0,x)=0
  \end{cases}
\end{align}
on the time interval $[0,T]$.

In order to solve \eqref{xc1}, we subdivide $[0,T]$ into finite subintervals such that on each subinterval, the influence of $v$ to problem \eqref{xc1} is small. Let $\eta$ be small constant. In view of \eqref{eneq1}, we can divide $[0,T]$ into subintervals $I_1,...,I_J$ such that on each $I_j=[t_j,t_{j+1}]$,
\beq\label{v-Ij-small}
\|v\|_{X_{I_j}^1}\sim \eta,\ \ 1\leq j\leq J.
\eeq
It's easy to get $J\leq C(\delta_0, \eta, \|u_0\|_2)$.

%Of course, we only concerned about the subintervals that have nonempty intersection with [0,T], so without loss of generality and renaming the intervals if necessary, we can write
%$$[0,T]=\bigcup_{j=1}^{J'}I_j,\ \ I_j=[t_j,t_{j+1}]$$with $J'\leq J$ and on each $I_j$, $\|v\|_{X_{I_j}^1}\lesssim \eta$.

 Now we aim to solve \eqref{xc1} by inductive arguments.
More precisely, we claim that for each $1\leq j\leq J$, %show that
 \eqref{xc1} has a unique solution $w$ on $I_j$ such that
\begin{align}\label{xc2}
  \|w\|_{L_t^\infty H^1(I_j\times \R^3)}+\|w\|_{X_{I_j}^1}\leq (2C)^jT^\frac12.
\end{align}
First, we assume \eqref{xc1} has been solved on $I_{j-1}$ and the solution $w$ satisfies the bound \eqref{xc2} for $j-1$. Then we only consider the problem on $I_j$.
Define the solution map $\Gamma$:
$$\Gamma w= e^{i(t-t_j)\Delta} w(t_j)-i\int_{t_j}^ t e^{i(t-s)\Delta}\left[|v+w|^2(v+w)-|v+w|^4(v+w)+|v|^4v\right](s)ds.$$
And  we will show that $\Gamma$ maps the complete set
$$\mathcal{B}=\{w:\|w\|_{L_t^\infty H^1(I_j\times \R^3)}+\|w\|_{X_{I_j}^1}\leq (2C)^jT^\frac12\}$$
into itself and is contractive under  the norm $\|\cdot\|_{X_{I_j}^0}$.

 Indeed, by  Strichartz, Sobolev,  and H\"older's inequality, we have
\begin{align*}
  \|\Gamma w\|_{L_t^\infty H^1(I_j\times \R^3)}+\|\Gamma w\|_{X_{I_j}^1}&\leq C\|w(t_i)\|_{H^1}+C\sum_{i=0}^{4}\|v\|_{X_{I_j}^1}^i\|w\|_{X_{I_j}^1}^{5-i}+CT^\frac12\|v+w\|_{X_{I_j}^1}^3\\
  &\leq C\|w(t_i)\|_{H^1}+C\sum_{i=0}^{4}\|v\|_{X_{I_j}^1}^i\|w\|_{X_{I_j}^1}^{5-i}+CT^\frac12\|w\|_{X_{I_j}^1}^3
  +CT^\frac12\eta^3,
\end{align*}
where $C$ is Strichartz constant.
Utilizing \eqref{v-Ij-small}  and our inductive assumption $\|w(t_j)\|_{H^1}\leq(2C)^{j-1}T^\frac12$, we see that for $w\in \mathcal{B}$,
\begin{align}\label{xce1}
  \|\Gamma w\|_{L_t^\infty H^1(I_j\times \R^3)}+\|\Gamma w\|_{X_{I_j}^1}\leq& C(2C)^{j-1}T^\frac12\\ \label{xce2}
  &+C(2C)^jT^\frac12\eta^4+CT^\frac12\eta^3\\\label{xce3}
  &+C\sum_{i=0}^{3}[(2C)^jT^\frac12]^{5-i}\eta^i+C((2C)^jT^\frac12)^3T^\frac12.
\end{align}
It is easy to observe that \eqref{xce1}$=\frac12(2C)^jT^\frac12$. We choose $\eta=\eta(C)$ and $T$ small enough such that $$\eqref{xce2} +\eqref{xce3} \leq \frac14(2C)^jT^\frac12.$$
 %Then fix $\eta$, choosing $T_0$ small so that
% $$\eqref{xce3}\leq\frac18(2C)^jT^\frac12$$
By the fact $J\leq C(\delta_0,\eta,\|u_0\|_2)$, we can choose $T$ %only depends on $\delta_0$, $M(u_0)$ and $\eta$ %Choosing T to be a small constant
%and
uniformly  of the process of induction. By the small token, for $w_1, w_2\in \mathcal{B}$, we also have
$$\|\Gamma w_1-\Gamma w_2\|_{X_{I_j}^0}\leq \frac12\|w_1-w_2\|_{X_{I_j}^0}.$$
From the fixed point theorem,  we can obtain a unique solution $w$ of \eqref{xc1} on $I_j$ which satisfies the bound \eqref{xc2}. Therefore, we get a unique solution of \eqref{xc1} on $[0,T]$ such that
$$\|w\|_{X_{[0,T]}}\leq \sum_{j=1}^{J}\|w\|_{X_{I_j}^1}\leq \sum_{j=1}^{J}(2C)^jT^\frac12\leq C(2C)^JT^\frac12\leq C(\delta_0,\eta,\|u_0\|_2).$$
Since the parameter  $\eta$ only depends on the Strichartz estimates, we have $u=v+w$ is the  solution to the Cauchy problem \eqref{Ct} on $[0,T]$ satisfying
$$\|u\|_{X_{[0,T]}^1}\leq \|v\|_{X_{[0,T]}^1}+\|w\|_{X_{[0,T]}^1}\leq C(\delta_0,\|u_0\|_2).$$

We briefly review some of variational analysis related to the ground state $W$. The ground state $W$ optimizes the sharp Sobolev  inequality: %Gagliardo-Nirenberg
$$\|f\|_6^6\leq C_3 \|\nabla f\|_2^6.$$
By a simple calculation, we deduce
\begin{align}\label{Var}
 \|\nabla W\|_2^2=\|W\|_6^6\ \  and \ \ C_3=\|\nabla W\|_2^{-4}.
\end{align}
From this and the ``global well-posedness'' above, we can deduce the following
 important property.

\begin{lemma}[Coercivity I]\label{gl}
If $E(u_0)\leq (1-\delta_0)E^c(W)$ and $\|\nabla u_0\|_2<\|\nabla W\|_2$, then there exists $\delta_1=\delta_1(\delta_0)>0$ so that
$$\|\nabla u(t)\|_2^2\leq(1-\delta_1)\|\nabla W\|_2^2.$$
for all $t\in I_{max}$.

In particular, $I_{max}=\R$ and $u(t)$ is uniformly bounded in $H^1$. And for any compact time interval $I\in \R$, we have
\begin{align}\label{zy}
\|u\|_{L_t^qW^{1,r}(I\times \R^3)}\leq |I|^\frac1qC(\delta_0,\|u_0\|_2).
\end{align}
\end{lemma}
\begin{proof}
By the energy conservation, we have
\begin{align}
  (1-\delta_0)(\frac12\|\nabla W\|_2^2-\frac16\|W\|_6^6)\geq \frac12\|\nabla u(t)\|_2^2-\frac16\|u(t)\|_6^6.
\end{align}
This  and the Sharp Sobolev inequality \eqref{Var} imply
\begin{align}
  (1-\delta_0)\geq \frac32 y(t)-\frac12y(t)^3,
\end{align}
where $y(t)=\frac{\|\nabla u(t)\|_2^2}{\|\nabla W\|_2^2}$.
By the fact $y(0)<1$ and continuity  arguments, there exists a constant $\delta_1<1$ such that
$$\|\nabla u(t)\|_2^2\leq(1-\delta_1)\|\nabla W\|_2^2,\ \ \forall t\in I_{max}.$$
This combines with  the ``good local well-posedness" above yield the global well-posedness of $u$.
Since $T$ is a fixed constant depending $\delta_0$, one can obtain $T$ is fixed constant.
%And as $\delta_0$ is fixed constant, then the $T$ is fixed constant. It is easy to get \eqref{zy}.
\end{proof}

\section{Scattering Criterion}

\noindent

 In \cite{Tao}, Tao established a scattering criterion for radial solution to energy-subcritical NLS. In this paper, we extend it to the  energy-critical  case.
\begin{lemma}[Scattering Criterion]\label{SC}
  Suppose $u:\R_t\times\R^3\rightarrow \mathbb{C}$ is a radial solution to \eqref{Ct} such that
  \begin{align}\label{zz0}
    \|u\|_{L_t^\infty H_x^1(\R\times\R^3)}\leq E.
  \end{align}
There exist $\epsilon=\epsilon(E)>0$ and $R=R(E)>0$ such that if
  \begin{align}\label{zz1}
   \liminf_{t\to \infty}\int_{|x|\leq R}|u(t,x)|^6dx\leq \epsilon ^6,
  \end{align}
then $u$ scatters forward in time.
\end{lemma}

\begin{proof}
Let $0<\epsilon <1$ and $R\geq 1 $ to be chosen later. By Sobolev embedding, Strichartz and convergence, we may choose $T$ large enough depending $u_0$ such that
\ben\label{be:2}
  \|e^{it\Delta}u_0\|_{L_{t}^{10}L_x^{10}([T,\infty)\times \R^3)} \leq \epsilon.
\een

By assumption \eqref{zz1}, we may choose $T_{0}>T$ such that
\begin{align}\label{ave:1}
\int_{\mathbb{R}^3} \chi_{R}|u(T_{0},x)|^6dt\leq \epsilon^6.
\end{align}

Using Duhamel formula to write
\begin{align}
  u(t)=e^{i(t-T_{0})\Delta}u(T_{0})-i\int_{T_{0}}^{t}e^{i(t-s)\Delta}(|u|^2u(s)-|u|^4u)(s)ds.
\end{align}
By standard continuity argument, Sobolev embedding and Strichartz estimates, we just need to show  that
\begin{align*}
  \|e^{i(t-T_{0})\Delta}u(T_{0})\|_{L_{t}^{10}L_x^{10}([T_{0},\infty)\times \R^3)}\ll 1.
\end{align*}
Noting that
\begin{align*}
  e^{i(t-T_{0})\Delta}u(T_{0})=e^{it\Delta}u_0-iF_{1}(t)-iF_{2}(t),
\end{align*}
where
\begin{align*}
  F_{j}(t):=\int_{I_{j}}e^{i(t-s)\Delta}(|u|^2u(s)-|u|^4u)(s)ds, ~j=1,2\quad \text{ and } \quad I_{1}=[0,T_{0}-\epsilon ^{-\theta}],\quad I_{2}=[T_{0}-\epsilon ^{-\theta},T_{0}].
\end{align*}
Then, by \eqref{be:2}, it remains to show
\begin{align*}
  \|F_{j}(t)\|_{L_{t}^{10}L_x^{10}([T_{0},\infty)\times \mathbb{R}^3)}\ll 1,\quad \text{~ for ~} j=1,2.
\end{align*}

\textbf{Estimation  of $F_{1}(t)$}:
It follows from the dispersive estimate, Young's inequality and the Sobolev embedding that
\begin{align*}
  \|F_{1}(t)\|_{L_{t}^{10}L_x^{10}([T_{0},\infty)\times \mathbb{R}^3)}&\lesssim \big\|\int_{0}^{T_{0}-\epsilon^{-\theta}}|t-s|^{-\frac{6}{5}}(\|u\|_{L^{\frac{30}{9}}}^3+\|u\|_{L^{\frac{50}{9}}}^5)ds
  \big\|_{L_{t}^{10}([T_{0},\infty))}\\
  &\lesssim \big\|\int_{0}^{T_{0}-\epsilon^{-\theta}}|t-s|^{-\frac{6}{5}}(E^3+E^5)ds\big\|_{L_{t}^{10}([T_{0},\infty))}\\
  &\lesssim \|(t-T_{0}+\epsilon^{-\theta})^{\frac{-1}{5}}\|_{L_{t}^{10}([T_{0},\infty))}\\
  &\lesssim \epsilon ^{\frac{\theta}{4}}.
\end{align*}
Hence, we have
\ben\label{bq:1}
  \|F_{1}(t)\|_{L_{t}^{10}L_x^{10}([T_{0},\infty)\times \mathbb{R}^3)}\lesssim \epsilon ^{\frac{\theta}{4}},
\een
which is sufficiently small when $\theta > 0$.

\textbf{Estimation  of $F_{2}(t)$}:
%
%Next, we consider $F_{2}(t)$:
%
From, H\"older's inequality, the Sobolev embedding, Lemma \ref{CR} and \eqref{zy}, one has for any interval $I$ that
 \begin{align*}
   \|u\|_{L_{t}^4L_x^\infty(I\times \mathbb{R}^3)}&\lesssim E+\|u\|_{L_t^2L_x^{\infty}}E+\|u\|_{L_t^8W_x^{1,\frac{12}{5}}}^4\|\nabla u\|_{L_t^2L_x^6} \\
   &\lesssim 1+|I|^\frac12+|I|\\
   &\lesssim 1+|I|.
 \end{align*}
By Sobolev, Strichartz and \eqref{zy}, we deduce
\begin{align*}
  \|F_{2}(t)\|_{L_{t}^{10}L_x^{10}([T_{0},\infty)\times \mathbb{R}^3)}\lesssim& \int_{T_{0}-\epsilon^{-\theta}}^{T_{0}} \|\big(|u|^2u-|u|^4u\big)\|_{H^1}ds\\
  \lesssim &\|u\|_{L_{t}^\infty L^3(I_{2}\times \mathbb{R}^3)}\|u\|_{L_{t}^2 L^\infty(I_{2}\times \mathbb{R}^3)}||u||_{L_{t}^2 W^{1,6}(I_{2}\times \mathbb{R}^3)}\\
  &+\|u\|_{L_t^\infty L_x^6(I_{2}\times \mathbb{R}^3)}^2\|u\|_{L_t^4L_x^\infty(I_{2}\times \mathbb{R}^3)}^2\|u\|_{L_t^2W^{1,6}(I_{2}\times \mathbb{R}^3)}\\
  \lesssim& \epsilon^{-\theta}\|u\|_{L_{t}^\infty L_x^3(I_{2}\times \mathbb{R}^3)}+\epsilon^{-\frac52\theta}\|u\|_{L_t^\infty L_x^6(I_{2}\times \mathbb{R}^3)}^2.
\end{align*}
From interpolation, it suffice to show
\begin{align*}
  \|u\|_{L_{t}^\infty L_x^6(I_{2}\times \mathbb{R}^3)}\ll 1.
\end{align*}
It is clear by H\"older's inequality that
\begin{align*}
\|u\|_{L_{t}^{\infty}L_{x}^{6}((I_{2}\times \mathbb{R}^3)}
~\leq ~&  \|\chi_{R}u\|_{L_{t}^{\infty}L_{x}^{6}}+\|(1-\chi_{R})u\|_{L_{t}^{\infty}L_{x}^{6}}\\
~\leq  ~& \|\chi_{R}u\|_{L_{t}^{\infty}L_{x}^{6}}
+\|(1-\chi_{R})u\|_{L_{t}^{\infty}L_{x}^{\infty}}^\frac{2}{3}
\|u\|_{L_{t}^{\infty}L_{x}^{2}}^\frac{1}{3}.
\end{align*}
Then, by  the radial Sobolev embedding Lemma \ref{lem2.1}, we deduce
\begin{align*}
\|(1-\chi_{R})u\|_{L_{t}^{\infty}L_{x}^{\infty}(I_2\times\mathbb{R}^3)}& \lesssim \frac{1}{R}\|(1-\chi_{R})|x|u\|_{L_{t}^{\infty}L_{x}^{\infty}(I_2\times \mathbb{R}^3)}\\
&\lesssim \frac{1}{R}\|u\|_{L_{t}^\infty H_{x}^1(I_2\times\mathbb{R}^3)}.
\end{align*}
Thus, we have  %it follows from Sobolev embedding and assumption \eqref{zz1} that
\begin{align*}
\|u\|_{L_{t}^{\infty}L_{x}^{6}(I_{2}\times \mathbb{R}^3)} \lesssim \|\chi_{R}u\|_{L_{t}^{\infty}L_{x}^{6}(I_{2}\times \mathbb{R}^3)}+R^{-\frac23}.
\end{align*}
Using identity $\partial_{t}|u|^6=-6\nabla \text{~Im~} (|u|^4\bar{u}\nabla u)$ and \eqref{ave:1}, together with integration by parts and H\"older's inequality, we have
\begin{align*}
\big| \partial_{t}{ \int_{\mathbb{R}^3} \chi_{R}|u|^6(t)} dx \big| \lesssim \big| \int_{\mathbb{R}^3} {\nabla \chi_{R}|u|^4\bar{u} \nabla u}dx \big|\lesssim \frac1R\|u\|_{ L_x^6}^5\|u\|_{W^{1,6}}\lesssim \frac{1}{R}\|u(t)\|_{W^{1,6}}.
\end{align*}
Thus, choosing $R>\epsilon^{\theta-6}$, we have
\begin{align*}
\|\chi_Ru\|_{L_{t}^\infty L_{x}^6(I_{1}\times \mathbb{R}^3)} \lesssim \epsilon.
\end{align*}
 Then we have
 $$ \|F_{2}(t)\|_{L_{t}^{10}L_x^{10}([T_{0},\infty)\times \mathbb{R}^3)}\lesssim \epsilon^{\frac12-\theta}+\epsilon^{2-\frac52\theta}.$$
Finally, let $\theta=\frac{1}{4}$, we have $\|F_{2}(t)\|_{L_{t}^{10}L_x^{10}([T_{0},\infty)\times \mathbb{R}^3)}\lesssim \epsilon^{\frac14}$.
Then we complete the proof of Lemma \ref{SC} by choosing $\epsilon$ is sufficient small.
\end{proof}

\section{Proof of Theorem \ref{recall}}

\noindent

Throughout this section, we suppose $u(t)$ is a solution to \eqref{Ct} satisfying the hypotheses of Theorem \ref{recall}. In particular, using the result of Section \eqref{a}, we have that $u$ is global and uniformly bounded in $H^1$, and that there exists $\delta > 0$ such that
\begin{align}\label{b}
  \sup_{t\in \R}\|\nabla u(t)\|_2^2~<~(1-2\delta)\|\nabla W\|_2^2
\end{align}

We will prove that the potential energy of energy-critical escapes to spatial infinity as $t\rightarrow \infty .$
\begin{proposition}\label{Pro}
  There exists a sequence of times $t_n \rightarrow \infty$ and of radii $R_n \rightarrow \infty$ such that
  $$\lim_{n\to\infty}\int_{|x|\leq R_n}|u(t_n,x)|^6dx=0.$$
\end{proposition}

Using Proposition \ref{Pro} and the scattering criterion above, we can quickly prove the first part of Theorem \ref{recall}. The other case is similar.

We prove Proposition \ref{Pro}, by a virial/Morawetz estimate. First, we need a lemma that gives \eqref{b} on large balls, so that we can exhibit the necessary coercivity. Let $\chi(x)$ be radial smooth function such that
 \begin{equation*}
  \chi(x)=\left\{
    \begin{aligned}
    &1,\ \  |x|\leq \frac12,\\
    &0,\ \  |x|>1.
 \end{aligned}\right.
 \end{equation*}
 Set $\chi_R(x):=\chi(\frac xR)$ for $R>0$.

 \begin{lemma}[Coercivity II]\label{QZ1}
  Suppose $\|\nabla f\|_2^2\leq (1-\delta_1)\|\nabla W\|_2^2$. Then there exist $\delta_2>0$ such that
  $$\|\nabla f\|_2^2-\|f\|_6^6\geq \delta_2 \|f\|_6^6.$$
\end{lemma}
\begin{proof}
  Write $$\|\nabla f\|_2^2-\|f\|_6^6=6E^c(f)-2\|\nabla f\|_2^2.$$
  By the sharp Sobolev inequality %Gagliardo-Nirenberg
  \begin{align}
    E^c(f)\geq& \frac12\|\nabla f\|_2^2[1-\frac13C_3\|\nabla f\|_2^4]\\
    \geq&\frac12[1-\frac13(1-\delta_1)]\\
  \geq  &\frac{2+\delta_1}{6}\|\nabla f\|_2^2.
  \end{align}
Thus, we have
$$\|\nabla f\|_2^2-\|f\|_6^6\geq \delta_1 \|\nabla f\|_2^2,$$ which implies
$$\|\nabla f\|_2^2-\|f\|_6^6\geq \frac{\delta_1}{1-\delta_1}\|\nabla f\|_2^2$$ as desired.
\end{proof}
\begin{lemma}[Coercivity on balls, \cite{BM}]\label{Cor}
  There exists $R=R(\delta,M(u),W)$ sufficiently large that
  $$\sup_{t\in\R}\|\nabla(\chi_Ru(t))\|_2^2\leq (1-\delta)\|\nabla W\|_2^2.$$
  In particular, there exists $\delta_3>0$ so that
   $$\|\nabla (\chi_Ru(t))\|_2^2-\|\chi_Ru(t)\|_6^6\geq \delta_3 \|\chi_Ru(t)\|_6^6.$$
\end{lemma}

\begin{lemma}[Morawetz identity]\label{lm:32} Let $a:\mathbb{R}^3\rightarrow \mathbb{R}$ be a smooth weight. Define
\begin{align}
M(t)=2\text{~Im~} \int \bar{u}\nabla u \cdot \nabla adx.
\end{align}
Then
\begin{align}
\frac{dM(t)}{dt}=-\int |u|^6\Delta a dx-\int |u|^2\Delta\Delta a dx+\int |u|^4\Delta a +4Re\int( \bar{u}_{i}a_{ij}u_{j})dx.
\end{align}
\end{lemma}

Let $R\gg 1$ to be chosen later. We take $a(x)$ to be a radial function satisfying

\begin{eqnarray}
a(x)=
\begin{cases}
|x|^2;& |x|\leq R\\
3R|x|;& |x>2R,
\end{cases}
\end{eqnarray}
and when $R<|x|\leq 2R$, there holds
\begin{align*}
  \partial_{r}a\geq 0,\partial_{rr}a\geq 0\quad and \quad |\partial^{\alpha}a| \lesssim R|x|^{-|\alpha|+1}.
\end{align*}
Here $\partial_{r}$ denotes the radial derivative. Under these conditions, the matrix $(a_{jk})$ is non-negative.
It is easy to verify that
\begin{eqnarray*}
\begin{cases}
a_{jk}=2\delta_{jk},\quad \Delta a=6,\quad \Delta \Delta a=0,& |x|\leq R,\\
a_{jk}=\frac{3R}{|x|}[\delta_{jk}-\frac{x_{j}x_{k}}{|x|^2}],\quad \Delta a=\frac{6R}{|x|},\quad \Delta \Delta a=0,& |x|>2R.
\end{cases}
\end{eqnarray*}
Thus, we can divide $\frac{dM(t)}{dt}$ as follows:
\begin{align}\label{eq:main}
\frac{dM(t)}{dt}=&8\int_{|x|\leq R} |\nabla u|^2-|u|^6+\frac{3}{4}|u|^4dx\\\label{eq:er1}
&+\int_{|x|>2R}\frac{-6R}{|x|}|u|^6+\int \frac{6R}{|x|}|u|^4+\frac{12R}{|x|}|\not\nabla u|^2 dx\\\label{eq:er2}
&+\int_{R<|x|\leq 2R}4Re\bar{u}_{i}a_{ij}u_{j}+\mathcal{O}(\frac{R}{|x|}|
u|^4+\frac{R}{|x|}|u|^6+\frac{R}{|x|^3}|u|^2)dx,
\end{align}
where $\not\nabla$ denotes the angular derivation,
subscripts denote partial derivatives, and repeated indices are summed in this paper.

Note that by Cauchy-Schwarz, \eqref{zz0}, and the choice of $a(x)$, we have
\begin{align*}
\|M(t)\|_{L_{t}^\infty}\lesssim_E R.
\end{align*}

\begin{proposition}[Virial/Morawetz estimates]\label{po:31}
 Let $T>0$, if $R=R(\delta_0,M(u),W)$ is sufficiently large, then
 \begin{align}\label{in3}
\frac{1}{T}\int_{0}^{T}\int_{|x|\leq R}|u(t,x)|^6dxdt\lesssim_{u,\delta_0} \frac{R}{T}+\frac{1}{R^2}.
\end{align}
\end{proposition}
\begin{proof}
Let $a(x)$ and M(t) be as in Lemma \ref{lm:32}. We consider (\ref{eq:main}) as the main term and (\ref{eq:er1}) and (\ref{eq:er2}) as error terms.
Using the identity
\begin{align*}
  \int\chi_{R}^2|\nabla u|^2dx=\int |\nabla(\chi_{R}u)|^2dx+\int \chi_{R}\Delta\chi_{R}|u|^2dx,
\end{align*}
we have
\begin{align*}
  \int\chi_{R}^2|\nabla u|^2dx-\int\chi_{R}^2|u|^6dx&=\int |\nabla(\chi_{R}u)|^2dx-\int|\chi_{R}u|^6dx\\
  &+\int \chi_{R}\Delta\chi_{R}|u|^2dx+\int(\chi_{R}^6-\chi_{R}^2)|u|^6dx.
\end{align*}
Thus, we deduce
\begin{align}
  (\ref{eq:main})&\geq
  \int |\nabla(\chi_{R}u)|^2dx-\int|\chi_{R}u|^6dx\\
  &+\int \chi_{R}\Delta\chi_{R}|u|^2dx+\int(\chi_{R}^6-1)|u|^6dx.
\end{align}
Using the Morawetz identity, we can deduce
\begin{align}
\int |\nabla(\chi_{R}u)|^2dx-\int|\chi_{R}u|^6dx\lesssim \frac{dM(t)}{dt} +\frac{1}{R^2}+\int_{|x|\geq R}[|u|^4+|u|^6]dx.
\end{align}

By Lemma \ref{Cor}, Lemma \ref{gl} and radial Sobolev embedding, we have
\begin{align*}
  \int_{|x|\geq R}[|u(t,x)|^4 + |u(t,x)|^6  ]dx \lesssim \frac{1}{R^2}+ \frac{1}{R^4},
\end{align*}
\begin{align*}
  \delta\int_{0}^{T}\int_{\R^3}|\chi_{R}  u |^6dxdt\lesssim_{E} R+\frac{T}{R^2},
\end{align*}
which ends the proof of Proposition \ref{po:31}.
\end{proof}

\emph{Proof of Proposition \ref{Pro}.}
Employing Proposition \ref{po:31} with $T$ sufficiently large and $R=T^\frac13$ we have
$$\frac1T\int_{0}^{T}\int_{|x|\leq T^\frac13}|u(t,x)|^6dx\lesssim T^{-2/3},$$
which suffices to give the desired result.

%
%\titleformat{\section}[display]
%{\normalfont\Large\bfseries}{Appendix~\Alph{section}.}{11pt}{\Large}
 \begin{appendix}
  \renewcommand{\thesection}{\Alph{section}.}
\section{Alternative characterization of $ \mathcal{K}^{+}$}

  \noindent

  In this appendix, we will show the alternative characterization of $\K^{\pm}$. Recall the definitions of  $\K^{\pm},\bar{\K}^{\pm}$:
  \begin{align*}
\K^+&=\{f \in H^1| f\  is\  radial,\ E(f)<E^c(W),\ K(f)\geq0 \},\\
\K^-&=\{f \in H^1| f\  is\  radial,\ E(f)<E^c(W),\ K(f)<0 \},\\
\bar{\K}^+&=\{f \in H^1|f\  is\  radial,\ E(f)<E^c(W),\ \|\nabla f\|_2^2\leq \|\nabla W\|_2^2\},\\
\bar{\K}^-&=\{f \in H^1|f\  is\  radial,\ E(f)<E^c(W),\ \|\nabla f\|_2^2 > \|\nabla W\|_2^2\},
\end{align*}
 where $K(f)~=~2(\|\nabla f\|_2^2-\|f\|_6^6)+\frac32\|f\|_4^4$. Let
  \begin{align*}
    H(f)=&\frac16(\|\nabla f\|_2^2+\|f\|_6^6),\\
    K^c(f)=&2\|\nabla f\|_2^2-2\|f\|_6^6.
  \end{align*}

  By the variational results of  Kenig-Merle  \cite{KM}, we know that the condition
  \begin{align*}
    \|\nabla f\|_2^2-\|f\|_6^6\geq 0,\ \ E^c(f)<E^c(W)
  \end{align*}
  is equivalent to the condition
  \begin{align*}
    \|\nabla f\|_2^2\leq \|\nabla W\|_2^2,\ \ E^c(f)<E^c(W).
  \end{align*}
This gives $\bar{\K}^+\subset \K^+$.

   Let $f_{1,-2}^\lambda=e^\lambda f(e^{2\lambda}x)$, we have  $f_{1,-2}^\lambda\in H^1$ and $f_{1,-2}^\lambda \neq 0$ for any $\lambda>0$. In addition, we have
   \begin{align*}
     K(f_{1,-2}^\lambda)&=2(\|\nabla f\|_2^2-\|f\|_6^6)+\frac32 e^{-3\lambda}\|f\|_4^4,\\
     H(f_{1,-2}^\lambda)&=H(f).
   \end{align*}
  By the facts in \cite{MXZ}
  \begin{align}
    &E^c(W)=\inf\{H(f)|f\in H^1,f\neq 0, K(f)\leq 0\},\\
    &E(f)=H(f)+\frac16 K(f),  \label{EHK}   \\
    &E(f_{1,-2}^\lambda)\leq E(f)<E^c(W), \text{ ~for~ } \lambda >0, \label{opp}
  \end{align}
 and
  $$K(f_{1,-2}^\lambda)=K^c(f)+\frac32 e^{-3\lambda}\|f\|_4^4,$$
we can deduce $\K^+\subset \bar{\K}^+$.
 In fact, if $K^c(f)< 0 $, there exists  $\lambda_0 \geq0$ such that $K(f_{1,-2})^{\lambda_0})=0 $ (since $K(f_{1,-2}^{0}\geq 0$). Then, by the identity \eqref{EHK},  we have
  $E(f_{1,-2}^{\lambda_0})=H(f)\geq E^c(W),$
which is opposite to \eqref{opp}.
Then, the result $\K^+= \bar{\K}^+$ follows, which implies $\K^-= \bar{\K}^-$.
  \end{appendix}

%\end{CJK*}
 \end{document}